\documentclass[11pt,a4paper,twoside]{article}
\usepackage{amsthm, amsfonts,amsmath}

\newtheorem{theorem}{Theorem}[section]
\newtheorem{corollary}{Corollary}[section]

\newtheorem{example}{Example}[section]
\newtheorem{lemma}{Lemma}[section]

\newtheorem{remark}{Remark}[section]

\author{
        Dhriti Sundar Patra\footnote{Department of Mathematics, Birla Institute of Technology Mesra, Ranchi: 835 215, India
       \newline e-mail: {\tt dhritimath@gmail.com} \ and \ {\tt dpatra.teqip@bitmesra.ac.in}}
        \ and \
       Vladimir Rovenski\footnote{Mathematical Department, University of Haifa, Mount Carmel, 31905 Haifa,  Israel
       \newline e-mail: {\tt vrovenski@univ.haifa.ac.il}
       }
}

\title{Almost $\eta$-Ricci solitons on Kenmotsu manifolds}

\begin{document}

\date{}

\maketitle

\begin{abstract}
In this paper we characterize the Einstein metrics in such broader classes of metrics
as almost $\eta$-Ricci solitons and $\eta$-Ricci solitons on Kenmotsu manifolds, and generalize some results of other authors.
First, we prove that a Kenmotsu metric as an $\eta$-Ricci soliton is
Einstein metric if either it is $\eta$-Einstein or the potential vector field $V$ is an infinitesimal contact transformation or $V$ is collinear to the Reeb vector field.
Further, we prove that if a Kenmotsu manifold admits a gradient almost $\eta$-Ricci soliton with a Reeb vector field leaving the scalar curvature invariant, then it is an Einstein manifold.
Finally, we present new examples of $\eta$-Ricci solitons and gradient $\eta$-Ricci solitons, which illustrate our results.


\end{abstract}




\section{Introduction}

\noindent
\textit{Ricci soliton}, defined on a Riemannian manifold $(M,g)$ by the partial differential equation
\begin{equation}\label{1.0}
 \frac{1}{2}\,{\cal L}_{V}\,g + {\rm Ric} + \lambda\,g = 0,
\end{equation}
is a natural generalization of Einstein metric (i.e., the Ricci tensor is a constant multiple of the Riemannian metric $g$).
In \eqref{1.0}, ${\cal L}_{V}$ denotes the Lie-derivative in the direction of $V$, ${\rm Ric}$ is the Ricci tensor of $g$ and $\lambda$ is a constant.
A Ricci soliton is trivial if $V$ is either zero or Killing on $M$.
A Ricci soliton is said to be shrinking, steady, and expanding, as $\lambda$ is negative, zero, and positive, respectively
(and there are many examples of each of them, e.g., \cite{ck1}). Otherwise, it will be called indefinite. Actually, a Ricci soliton can be considered as a generalized fixed point of the normalized version of Hamilton's Ricci flow \cite{hamilton1988ricci}: $\frac{\partial g}{\partial t}=-2{\rm Ric}$, viewed as a dynamical system in the quotient space of Riemannian metrics modulo diffeomorphisms and rescalings, e.g., \cite{ck1}.
In \cite{pigola2011ricci}, Pigoli et. al. assumed the soliton constant $\lambda$ to be a $C^\infty$-regular (i.e., smooth)  function and named it as Ricci almost soliton, later on studied by Barros et. al. \cite{barros2012some}.
Some applications of (semi-)Riemannian almost Ricci solitons to general relativity were discussed in \cite{duggal2017almost}.
In~\cite{cho2009ricci}, Cho-Kimura studied real hypersurfaces in a complex space form and
 generalized the notion of Ricci soliton to $\eta$-\textit{Ricci soliton}, defined on
 $(M,g)$ by
\begin{equation}\label{1.1}
 \frac{1}{2}\,{\cal L}_{V}\,g + {\rm Ric} + \lambda\,g + \mu\,\eta \otimes \eta=0,
\end{equation}
where the tensor product notation $(\eta\otimes\eta)(X,Y)=\eta(X)\eta(Y)$ is used
and $\lambda$, $\mu$ are real constants. Later on \eqref{1.1} was studied by Calin-Crasmareanu \cite {calin2012eta}, Blaga et. al. \cite{blaga2018almost,blaga2015eta,blaga2020almost} and Naik-Venkatesha \cite{naik2019eta}.
An~$\eta$-Ricci soliton is said to be almost $\eta$-Ricci soliton if
$\lambda$ and $\mu$ are smooth functions on $M$ (see details in \cite{blaga2018almost,blaga2020almost}).
When the potential vector field $V$ is a gradient of a smooth function $f : M \rightarrow \mathbb{R}$ (called the potential function) the manifold will be called a gradient almost $\eta$-Ricci soliton, and
\eqref{1.2} reads~as
\begin{equation}\label{1.2}
 {\rm Hess}~f + {\rm Ric} + \lambda\,g + \mu\,\eta\otimes\eta=0,
\end{equation}
where ${\rm Hess}\,f$ is the Hessian of $f$.
As a generalization of a Ricci soliton, it is also said to be shrinking, steady, and expanding, as $\lambda$ is negative, zero, and positive, respectively, e.g., \cite{blaga2018almost,naik2019eta}.

Many authors studied Ricci solitons, $\eta$-Ricci solitons and their generalizations in the framework of almost contact and paracontact geometries,
e.g., contact metrics as Ricci solitons by Cho-Sharma \cite{cho2010contact}, $K$-contact and $(k,\mu)$-contact metrics as Ricci solitons by Sharma \cite{sharma2008certain}, contact metrics as Ricci almost solitons by Ghosh-Sharma \cite{ghosh2014certain,sharma2014almost}, contact metrics as $h$-almost Ricci solitons by Ghosh-Patra \cite{ghosh2018k},
almost contact $B$-metrics as Ricci-like soliton by Manev \cite{M20},
para-Sasakian and Lorentzian para-Sasakian metrics as $\eta$-Ricci solitons by Naik-Venkatesha \cite{naik2019eta} and Blaga \cite{blaga2015eta}, 
etc.
Based on the above results in a modern and active field of research, a natural \textbf{question} can be~posed:

\smallskip

\textit{Are there almost contact metric manifolds, whose metrics are $\eta$-Ricci solitons?}

\smallskip\noindent
The notion of a warped product is very popular in differential geometry as well as in general relativity, e.g.,~\cite{chen-2017}.
Some solutions of Einstein field equations are warped products and some spacetime models, e.g., Robertson-Walker spacetime, Schwarzschild spacetime, Reissner-Nordstr\"{o}m spacetime and asymptotically flat spacetime, are warped product manifolds
(see \cite{HE73} for more details).
In \cite{kenmotsu1972class}, Kenmotsu first introduced and studied a special class of
almost contact metric manifolds, known as Kenmotsu manifolds, which is characterized in terms of warped products.
The warped product $\mathbb{R}\times_{f} N$ (of the real line $\mathbb{R}$ and a K\"{a}hler manifold $N$) with the warping function
\begin{equation}\label{E-f-warp}
 f(t)=ce^t,
\end{equation}
given on an open interval $J=(-\varepsilon, \varepsilon)$,
admits Kenmotsu structure, and conversely, every point of a Kenmotsu manifold has a neighbourhood, which is locally a warped product $J\times_{f} N$, where $f$ is given by \eqref{E-f-warp}.
First, Ghosh considered an almost contact metric, in particular, a Kenmotsu metric, as a Ricci soliton and proved that a $3$-dimensional Kenmotsu metric as a Ricci soliton is of constant negative curvature $-1$ in \cite{ghosh2011kenmotsu}, and for higher dimension, a Kenmotsu metric as a Ricci soliton is Einstein if the metric is $\eta$-Einstein \cite{ghosh2013eta} or the potential vector field $V$ is contact, see~\cite{ghosh2019ricci}.
Shanmukha-Venkatesha \cite{SV20} studied Ricci semi-symmetric
Kenmotsu manifolds with $\eta$-Ricci solitons, and Sabina et. al. \cite{ECB18} studied $\eta$-Ricci solitons on a Kenmotsu manifold satisfying some curvature conditions.


Our goal in this paper is to answer the question posed above using methods of local Riemannian geometry and to characterize Einstein metrics in some of the broader classes of metrics mentioned above. Thus, we study almost $\eta$-Ricci solitons, in particular, $\eta$-Ricci solitons, on a Kenmotsu manifold.
 This paper is organized as follows. In Section 2, the basic facts about contact metric manifolds and Kenmotsu manifolds are given.
In Section 3, we consider Kenmotsu metrics as $\eta$-Ricci solitons and find some important conditions, when a Kenmotsu metric as an $\eta$-Ricci soliton is Einstein, and Theorems~\ref{thm3.1} and \ref{thm3.3} and Corollary~\ref{thm3.2} generalize some results of the above mentioned authors.
In Section 4, we consider almost $\eta$-Ricci solitons on a Kenmotsu manifold and find some $\eta$-Einstein and Einstein manifolds, using concept of almost $\eta$-Ricci solitons. We also present examples of Kenmotsu manifolds that admit $\eta$-Ricci solitons and gradient $\eta$-Ricci solitons,
which illustrate our results.

\section{Notes on contact metric manifolds}

Here, we give some definitions and basic facts (see details in \cite{blair2010riemannian,kenmotsu1972class}), which are used in the paper. An~\textit{almost contact structure} on a smooth manifold $M^{2n+1}$ of dimension $2\,n+1$ is a triple $(\varphi,\xi,\eta)$, where $\varphi$ is a $(1,1)$-tensor, $\xi$ is a vector field (called Reeb vector field) and $\eta$ is a 1-form, satisfying
\begin{equation}\label{2.1}
 \varphi^2 = -I + \eta\otimes \xi,\quad \eta(\xi) = 1,
\end{equation}
where $I$ denotes the identity endomorphism. It follows from \eqref{2.1} that $\varphi(\xi)=0$, $\eta\circ\varphi=0$ and ${\rm rank}\,\varphi=2\,n$ (see \cite{blair2010riemannian}). A smooth manifold $M$ endowed with an almost contact structure is called an almost contact manifold.
A Riemannian metric $g$ on $M$ is said to be {compatible} with an almost contact structure $(\varphi,\xi,\eta)$ if
\[
 g(\varphi X, \varphi Y) = g(X,Y) - \eta(X)\,\eta(Y)
\]
for any $X\in\mathfrak{X}(M)$, where $\mathfrak{X}(M)$ is the Lie algebra of all vector fields on $M$.
An almost contact manifold endowed with a compatible Riemannian metric is said to be an \textit{almost contact metric manifold} and is denoted by $M(\varphi,\xi,\eta,g)$.
The {fundamental $2$-form} $\Phi$ on $M(\varphi,\xi,\eta,g)$ is defined by $\Phi(X,Y)=g(X,\varphi Y)$ for any $X\in\mathfrak{X}(M)$.
An almost contact metric manifold satisfying $d\eta = 0$ and $d\Phi = 2\,\eta\wedge\Phi$ is said to be an almost Kenmotsu manifold (e.g.,~\cite{dileo2007almost,dileo2009almost}).
An almost contact metric structure is said to be normal if the tensor $N_{\varphi}=[\varphi,\varphi] + 2\,d\,\eta\otimes \xi$ vanishes on $M$,
where $[\varphi,\varphi]$ denotes the Nijenhuis tensor of $\varphi$.
A normal almost Kenmotsu manifold is said to be a Kenmotsu manifold (see \cite{kenmotsu1972class}); equivalently, an almost contact metric manifold is said to be a Kenmotsu manifold (see \cite{kenmotsu1972class}) if
\begin{equation*}
 (\nabla_{X}\,\varphi)Y=g(\varphi X,Y)\,\xi -\eta(Y)\,\varphi X
\end{equation*}
for any $X,\,Y\in\mathfrak{X}(M)$, where $\nabla$ is the Levi-Civita connection of $g$.
The following formulas hold on Kenmotsu manifolds, see~\cite{ghosh2019ricci,kenmotsu1972class}:
\begin{eqnarray}\label{2.3}
 && \nabla_{X}\,\xi = X - \eta(X)\xi,\\
\label{2.4}
 && R(X, Y)\xi = \eta(X)Y - \eta(Y)X,\\
\label{3.1}
 && (\nabla_{\xi}Q)X=-2\,QX -4n X,\\
\label{2.5}
 && Q\xi = -2n\xi,
\end{eqnarray}
for all $X,Y\in\mathfrak{X}(M)$, where $R$ and $Q$ denote, respectively, the curvature tensor and the Ricci operator of $g$ associated with the Ricci tensor and given by
 ${\rm Ric}(X,Y)=g(QX,Y)$ for
 $X,Y\in\mathfrak{X}(M)$.
Note that \eqref{2.3}--\eqref{3.1} are simple and \eqref{2.5} is proven in \cite{ghosh2019ricci}.
Recall the commutation formula, see p.~23 in \cite{yano1970integral},
\begin{equation}\label{2.6}
 ({\cal L}_{V}\,\nabla_{Z}\,g - \nabla_{Z}{\cal L}_{V}\,g - \nabla_{[V,Z]}\,g)(X,Y) = -g(({\cal L}_{V}\,\nabla)(Z,X),Y) -g(({\cal L}_{V}\,\nabla)(Z,Y),X).
\end{equation}
A contact metric manifold $M^{2n+1}$ is called $\eta$-\textit{Einstein}, if its Ricci tensor has the following form:
\begin{equation}\label{3.15}
 {\rm Ric} = \alpha\,g + \beta\,\eta\otimes\eta,
\end{equation}
where $\alpha$ and $\beta$ are smooth functions on $M$.
For an $\eta$-Einstein $K$-contact manifold of dimension $>3$, these $\alpha$ and $\beta$ are constant~\cite{yano1984structures}), but for an $\eta$-Einstein Kenmotsu manifold this is not true~\cite{kenmotsu1972class}.

\section{On $\eta$-Ricci solitons}

Here,
we study $\eta$-Ricci solitons on a Kenmotsu manifold.


\begin{lemma}\label{lem3.2}
Let $M^{2n+1}(\varphi,\xi,\eta,g)$ be a Kenmotsu manifold. If $g$ represents an $\eta$-Ricci soliton with the potential vector field $V$ then  $({\cal L}_V\,R)(X,\xi)\xi = 0$ for all $X\in\mathfrak{X}(M)$.
\end{lemma}

\begin{proof} Taking the covariant derivative of (\ref{1.1}) along $Z\in\mathfrak{X}(M)$ and using \eqref{2.3}, we have
\begin{equation}\label{3.3A}
 (\nabla_Z\,{\cal L}_{V}\,g)(X,Y) = -2(\nabla_{Z}{\rm Ric})(X,Y)-2\mu\{g(X, Z)\eta(Y) + g(Y,Z)\eta(X) -2\eta(X)\eta(Y)\eta(Z)\}
\end{equation}
for all $X,Y\in\mathfrak{X}(M)$. Since Riemannian metric is parallel, it follows from \eqref{2.6} that
\begin{equation*}
 (\nabla_{Z}\,{\cal L}_{V}\,g)(X,Y) = g(({\cal L}_{V}\,\nabla)(Z,X),Y) + g(({\cal L}_{V}\,\nabla)(Z,Y),X).
\end{equation*}
Plugging it into (\ref{3.3A}), we obtain
\begin{eqnarray}\label{3.5}
\nonumber
 && g(({\cal L}_{V}\,\nabla)(Z,X),Y) + g(({\cal L}_{V}\,\nabla)(Z,Y),X) = -2(\nabla_{Z}{\rm Ric})(X,Y)\\
 &&-2\mu\{g(X, Z)\eta(Y) + g(Y,Z)\eta(X)-2\eta(X)\eta(Y)\eta(Z)\}
\end{eqnarray}
for all $X,Y,Z\in\mathfrak{X}(M)$. Cyclically rearranging $X,Y$ and $Z$ in \eqref{3.5}, we obtain
\begin{eqnarray}\label{3.6}
\nonumber
 g(({\cal L}_{V}\,\nabla)(X,Y),Z) &=& (\nabla_{Z}{\rm Ric})(X,Y)-(\nabla_{X}{\rm Ric})(Y,Z)-(\nabla_{Y}{\rm Ric})(Z,X)\\
 &-& 2\mu\{g(X, Y)\eta(Z)-\eta(X)\eta(Y)\eta(Z) \}.
\end{eqnarray}
Taking the covariant derivative of \eqref{2.5} along $X\in\mathfrak{X}(M)$ and using \eqref{2.3}, we obtain
\begin{equation}\label{3.2}
 (\nabla_{X}Q)\xi=-QX -2n X.
\end{equation}
Substituting $\xi$ for $Y\in\mathfrak{X}(M)$ in \eqref{3.6} and applying \eqref{3.1} and \eqref{3.2}, we obtain
\begin{equation}\label{3.7}
 ({\cal L}_{V}\,\nabla)(X,\xi) = 2Q X + 4n X
\end{equation}
for any $X\in\mathfrak{X}(M)$. Next, using (\ref{2.3}), (\ref{3.7}) in the covariant derivative of (\ref{3.7}) along $Y$, yields
\begin{eqnarray*}
 (\nabla_Y {\cal L}_V\,\nabla)(X, \xi) + ({\cal L}_V\,\nabla)(X, Y) =2(\nabla_YQ)X+ 2\eta(Y)(QX + 2n X)
\end{eqnarray*}
for any $X\in\mathfrak{X}(M)$. Plugging this in the following commutation formula (see \cite{yano1970integral}, p.~23):
\begin{equation*}
 ({\cal L}_{V}\,R)(X,Y)Z = (\nabla_{X}{\cal L}_{V}\,\nabla)(Y,Z) -(\nabla_{Y}{\cal L}_{V}\,\nabla)(X,Z),
\end{equation*}
we deduce
\begin{eqnarray}\label{3.9}
 ({\cal L}_V\,R)(X,Y)\xi = 2\{(\nabla_XQ)Y - (\nabla_YQ)X\}
 + 2\{\eta(X)QY - \eta(Y)QX\} + 4n\{\eta(X)Y - \eta(Y)X\}
\end{eqnarray}
for all $X,Y\in\mathfrak{X}(M)$. Substituting $Y$ by $\xi$ in \eqref{3.9} and using \eqref{3.1}, \eqref{2.5} and \eqref{3.2}, yields our result.
\end{proof}

Now, we consider an $\eta$-Einstein Kenmotsu metric as an $\eta$-Ricci soliton and characterize the Einstein metrics in such a wider class of metrics.

\begin{theorem}\label{thm3.1}
Let $M^{2n+1}(\varphi, \xi, \eta, g)$, $n>1$, be an $\eta$-Einstein Kenmotsu manifold. If $g$ represents an $\eta$-Ricci soliton with the potential vector field $V$, then it is
Einstein with constant scalar curvature $r=-2n(2n+1)$.
\end{theorem}

\begin{proof} First, tracing \eqref{3.15} gives $r=(2n+1)\alpha+\beta$ and  putting $X=Y=\xi$ in \eqref{3.15} and using \eqref{2.5}, yields $\alpha+\beta=-2n$ and $r=2n(\alpha-1)$. Therefore, by computation, \eqref{3.15} can be transformed into
\begin{equation*}
 {\rm Ric} = \big(1+\frac{r}{2n}\big)\,g -\big(2n+1+\frac{r}{2n}\big)\,\eta\otimes\eta.
\end{equation*}
By \eqref{2.1}, the foregoing equation entails that
\begin{equation*}
 (\nabla_Y Q)X =\frac{Y(r)}{2n}\varphi^2 X -\big(2n+1+\frac{r}{2n}\big)\big( g(X,Y)\xi + \eta(X)(Y-2\eta(Y)\xi)\big)
\end{equation*}
for all $X,Y\in\mathfrak{X}(M)$. By virtue of this, \eqref{3.9} provides
\begin{equation}\label{3.17}
 ({\cal L}_V\,R)(X, Y)\xi = \big(X(r)\varphi^2 Y-Y(r)\varphi^2 X\big)/n
\end{equation}
for all $X,Y\in\mathfrak{X}(M)$. Inserting $Y=\xi$ in \eqref{3.17} and applying Lemma~\ref{lem3.2}, we get $\xi(r)\varphi^2 X=0$ for any $X\in\mathfrak{X}(M)$. This implies $\xi(r)=0$. Using this in the trace of \eqref{3.2}, gives $r=-2n(2n+1)$.
\end{proof}

The curvature tensor of a $3$-dimensional Riemannian manifold has well-known form
\begin{equation}\label{3.18}
  R(X,Y)Z = g(Y,Z)QX -g(X,Z)QY +g(QY,Z)X -g(QX,Z)Y -\frac{r}{2}\big(g(Y,Z)X -g(X,Z)Y\big)
\end{equation}
for all $X,Y,Z\in\mathfrak{X}(M)$. Inserting $Y=Z=\xi$ in \eqref{3.18} and using \eqref{2.5}, we obtain
\begin{equation}\label{3.19}
 QX=(1+\frac{r}{2})X -\big( 3+\frac{r}{2}\big)\eta(X)\xi
\end{equation}
for any $X\in\mathfrak{X}(M)$. Thus, proceeding in the same way as in Theorem~\ref{thm3.1} and applying Lemma~\ref{lem3.2} we conclude that $r=-6$.
Hence, from \eqref{3.19} we have $QX=2nX$. Plugging this in \eqref{3.18} implies that $(M,g)$ is of constant negative curvature $-1$. Thus, from Theorem~\ref{thm3.1} we obtain the following.

\begin{corollary}\label{thm3.2}
If a $3$-dimensional Kenmutsu metric $g$ represents an $\eta$-Ricci soliton with the potential vector field $V$, then it is of constant negative curvature $-1$.
\end{corollary}

\begin{remark}\rm
Ghosh proved Theorem~\ref{thm3.1} in \cite{ghosh2013eta} and Corollary~\ref{thm3.2} in \cite{ghosh2011kenmotsu} for Ricci solitons.
In~this article, using different technique, we prove two above results in a short and direct way for an $\eta$-Ricci soliton on a Kenmotsu manifold. Since Ricci soliton is a particular case of an $\eta$-Ricci soliton, we can also derive similar results for Ricci solitons using this method.
\end{remark}

A vector field $X$ on a contact metric manifold $M$ is called a \textit{contact} or \textit{infinitesimal contact transformation} if it preserves the contact form $\eta$, i.e., there is a smooth function $\rho:M\to\mathbb{R}$ satisfying
\begin{equation}\label{3.20}
 {\cal L}_{X}\,\eta=\rho\,\eta,
\end{equation}
and if $\rho=0$ then the vector field $X$ is called strict.
We
consider a Kenmotsu metric as an $\eta$-Ricci soliton, whose potential vector field $V$ is contact or $V$ collinear to $\xi$.
First, we derive the following.

\begin{lemma}\label{lem3.3}
Let $M^{2n+1}(\varphi,\xi,\eta,g)$ be a Kenmotsu manifold. If $g$ represents an $\eta$-Ricci soliton with a potential vector field $V$ then we have $\lambda+\mu=2n$.
\end{lemma}

\begin{proof} Equation (\ref{2.4}) gives us $R(X, \xi)\xi=-X+\eta(X)\xi$. Taking its Lie derivative along $V$, we obtain
\begin{equation}\label{3.11}
 ({\cal L}_V\,R)(X, \xi)\xi + R(X, {\cal L}_V\,\xi)\xi + R(X, \xi){\cal L}_V\,\xi = \{({\cal L}_V\,\eta)X\}\xi + \eta(X){\cal L}_V\,\xi.
\end{equation}
Putting $({\cal L}_V\,R)(X, \xi)\xi=0$ (by Lemma~\ref{lem3.2}) and applying \eqref{2.4}, equation \eqref{3.11} becomes
\begin{equation}\label{3.12}
 ({\cal L}_V\,g)(X,\xi) + 2\,\eta({\cal L}_V\,\xi)X=0
\end{equation}
for any $X\in\mathfrak{X}(M)$.
Using \eqref{1.1} in the Lie derivative of $g(\xi, \xi) =1$, we obtain
\begin{equation}\label{3.14}
 \eta({\cal L}_V\,\xi) = \lambda+\mu-2\,n.
\end{equation}
Plugging the values of $({\cal L}_V\,g)(X,\xi)$ and $\eta({\cal L}_V\,\xi)$ from \eqref{1.1} and \eqref{3.14} respectively, into \eqref{3.12}, we obtain $(2n-\lambda-\mu)\varphi^2X=0$.
Using \eqref{2.1} and then tracing, we achieve the required result.
\end{proof}

\begin{theorem}\label{thm3.3}
Let $M^{2n+1}(\varphi, \xi, \eta, g)$, $n>1$, be a Kenmotsu manifold. If $g$ represents an $\eta$-Ricci soliton and the potential vector field $V$ is an infinitesimal contact transformation, then $V$ is strict and $(M,g)$ is an Einstein manifold with constant scalar curvature $r=-2n(2n+1)$.
\end{theorem}

\begin{proof} First, recall the known formula (see p.~23 of \cite{yano1970integral}):
\begin{equation}\label{3.21}
 {\cal L}_{V}\nabla_{X}\,Y - \nabla_{X}\,{\cal L}_{V}\,Y - \nabla_{[V,X]}Y = ({\cal L}_{V}\,\nabla)(X, Y)
\end{equation}
for all $X,Y\in\mathfrak{X}(M)$. Now, setting $Y=\xi$ in \eqref{3.21} and using \eqref{2.3}, we find
\begin{equation}\label{3.22}
 ({\cal L}_V\,\nabla)(X, \xi) = {\cal L}_V(X-\eta(X)\xi)-{\cal L}_V\,X+\eta({\cal L}_V\,X)\xi
 =-({\cal L}_V\,\eta)(X)\xi - \eta(X){\cal L}_V\,\xi
\end{equation}
for any $X\in\mathfrak{X}(M)$. Taking Lie-derivative of $\eta(X)=g(X,\xi)$ along $V$ and using \eqref{1.1}, \eqref{3.20} and Lemma~\ref{lem3.3}, we obtain ${\cal L}_V\,\xi=\rho\xi$. Thus, \eqref{3.14} and Lemma~\ref{lem3.3} gives $\rho=0$, therefore, ${\cal L}_V\,\xi=0$ and $V$ is strict. Also, \eqref{3.20} gives ${\cal L}_{V}\,\eta=0$. It follows from \eqref{3.22} that $({\cal L}_V\,\nabla)(X, \xi) =0$. Thus, from \eqref{3.7} we prove the rest part of this theorem.
\end{proof}

\begin{theorem}
Let $M^{2n+1}(\varphi, \xi, \eta, g)$ be a Kenmotsu manifold. If $g$ represents a $\eta$-Ricci soliton with non-zero potential vector field $V$ is collinear to $\xi$, then $g$ is Einstein with constant scalar curvature $r=-2n(2n+1)$.
\end{theorem}

\begin{proof}  By hypothesis: $V=\sigma\,\xi$ for some smooth function $\sigma$ on $M$. Using \eqref{2.3} in the covariant derivative of $V=\sigma\,\xi$ along arbitrary $X\in\mathfrak{X}(M)$ yields
$\nabla_X V = X(\sigma)\xi +\sigma (X-\eta(X)\xi)$
for any $X\in\mathfrak{X}(M)$. Taking this into account, the soliton equation \eqref{1.1} transforms into
\begin{equation}\label{3.23}
 2{\rm Ric}(X,Y)+X(\sigma)\eta(Y)+Y(\sigma)\eta(X) +2(\sigma+\lambda) g(X,Y)-2(\sigma-\mu)\eta(X)\eta(Y)=0
\end{equation}
for all $X,Y\in\mathfrak{X}(M)$. Inserting $X=Y=\xi$ in \eqref{3.23} and using \eqref{2.5} and Lemma~\ref{lem3.3}, we get $\xi(\sigma)=0$.
It~follows from \eqref{3.23} that $X(\sigma)=0$. Putting it into \eqref{3.23} gives
\begin{equation}\label{3.24}
 {\rm Ric}= -(\sigma+\lambda)\,g +(\sigma-\mu)\,\eta\otimes\eta.
\end{equation}
This shows that $(M,g)$ is an $\eta$-Einstein manifold, therefore, from Theorem~\ref{thm3.1} we conclude that $(M,g)$ is an Einstein manifold. Thus, from \eqref{3.24} we have $\sigma=\mu$, therefore, $\sigma+\lambda=2n$ (follows from Lemma~\ref{lem3.3}). This implies, using  \eqref{3.24}, that ${\rm Ric}=-2ng$, hence, $r=-2n(2n+1)$, as required.
\end{proof}

\begin{remark}\rm
In the above theorem we got $X(\sigma)=0$ for all vector fields $X$, therefore, the smooth function $\sigma$ reduces to a constant and it equals to the constant soliton $\mu$, hence, $V=\mu\,\xi$.
\end{remark}


In \cite{kenmotsu1972class}, Kenmotsu found the following necessary and sufficient condition for a Kenmotsu manifold $M^n$ to have constant
$\varphi$-holomorphic sectional curvature $H$:
\begin{eqnarray}\label{3.25M}
\nonumber
 4R(X,Y)Z=(H-3)\{g(Y,Z)X-g(X,Z)Y\}+(H+1)\{\eta(X)\eta(Z)Y-\eta(Y)\eta(Z)X\\
 +\eta(Y)g(X,Z)\xi-\eta(X)g(Y,Z)\xi+g(X,\varphi Z)\varphi Y-g(Y,\varphi Z)\varphi X+2g(X,\varphi Y)\varphi Z\}
\end{eqnarray}
for all $X,Y,Z\in\mathfrak{X}(M)$.
The following corollary of Theorem~\ref{thm3.1} illustrates Lemma~\ref{lem3.3}
and gives an example of a Kenmotsu manifold that admits an $\eta$-Ricci soliton.

\begin{corollary}
If a metric $g$ of a warped product $\mathbb{R}\times_{f} N^{2n}$ for $n>1$,
where $f(t)=ce^t$
and $N$ is a K\"{a}hler manifold
represents an $\eta$-Ricci soliton with the potential vector field $V$, then it is of constant
curvature $H=-1$.
\end{corollary}

\begin{proof}
Idea of the following construction is that if we choose $V=n\,\xi$, i.e., $\sigma=n$, then $\mathbb{R}\times_{f} N^{2n}$ is an $\eta$-Ricci soliton for $n>1$ with $\lambda=n$ and $\mu=n$.
It follows from \eqref{3.25M} that
\begin{equation}\label{3.26M}
 4QX =((2n-3)H-3(2n+1)) X -(2n-3)(H+1)\eta(X)\xi
\end{equation}
for any $X\in\mathfrak{X}(M)$. Also, a warped product $\mathbb{R}\times_{f} N^{2n}$, where $f(t)=ce^t$ on the real line $\mathbb{R}$, see \eqref{E-f-warp}, and $N^{2n}$ is a K\"{a}hler manifold, admits the Kenmotsu structure.
Thus, it follows from \eqref{3.26M} that $\mathbb{R}\times_{f} N^{2n}$ is an $\eta$-Einstein Kenmotsu manifold, and by Theorem~\ref{thm3.1}, we conclude that it is an Einstein manifold.
Thus, it follows from \eqref{3.26M} that $\beta=0$, i.e., $H=-1$, hence, $\alpha=-2n$.
\end{proof}

\section{On almost $\eta$-Ricci solitons}

Here, we study almost $\eta$-Ricci solitons on a Kenmotsu manifold. It is known that an almost $\eta$-Ricci soliton (i.e., satisfying \eqref{1.1} for some smooth functions $\lambda$ and $\mu$) is the generalization of a Ricci almost soliton.
In \cite{sharma2008certain}, Sharma proved that if a $K$-contact metric represents a gradient Ricci soliton, then it is an Einstein manifold, and Ghosh \cite{ghosh2014certain} generalized this result for Ricci almost solitons. Recently, Ghosh \cite{ghosh2019ricci} considered Ricci almost solitons on a Kenmotsu manifold and proved that if a Kenmotsu metric is a gradient Ricci almost soliton and the Reeb vector field $\xi$ leaves the scalar curvature $r$ invariant, then it is an Einstein manifold. To generalize the above results, we consider gradient almost $\eta$-Ricci solitons on a Kenmotsu manifold and prove the following.

\begin{theorem}\label{thm4.1}
If a Kenmotsu manifold $M^{2n+1}(\varphi, \xi, \eta, g)$ admits a gradient almost $\eta$-Ricci soliton and
$\xi$ leaves the scalar curvature $r$ invariant, then $(M,g)$ is an Einstein manifold with constant scalar curvature $r=-2n(2n+1)$.
\end{theorem}

\begin{proof} The gradient version of the soliton equation \eqref{1.2} can be exhibited
for any $X\in\mathfrak{X}(M)$ as
\begin{equation}\label{3.24k}
 \nabla_XDf+QX+\lambda X+\mu\eta(X)\xi=0.
\end{equation}
The curvature tensor, obtained from \eqref{3.24k} and the definition $R(X,Y)=[\nabla_{X},\nabla_{Y}]-\nabla_{[X,Y]}$, satisfies
\begin{eqnarray}\label{3.25}
\nonumber
 R(X,Y)Df &=& (\nabla_{Y}Q)X - (\nabla_{X}Q)Y + Y(\lambda)X-X(\lambda)Y\\
 &+& Y(\mu)\eta(X)\xi - X(\mu)\eta(Y)\xi+\mu\big(\eta(Y)X-\eta(X)Y\big)
\end{eqnarray}
for all $X,Y\in\mathfrak{X}(M)$. Now, replacing $Y$ by $\xi$ in \eqref{3.25} and using \eqref{3.1} and \eqref{3.2}, we obtain
\begin{equation}\label{3.27A}
 R(X,\xi)Df=-QX-2nX+\xi(\lambda)X-X(\lambda)\xi-X(\mu)\xi+\xi(\mu)\eta(X)\xi +\mu\varphi^2X
\end{equation}
for any $X\in\mathfrak{X}(M)$. By virtue of \eqref{2.4}, equation \eqref{3.27A} reduces to
\begin{equation}\label{3.27KK}
 X(\lambda+\mu+f)\xi =-QX +\big(\xi(\lambda+f)+\mu-2n\big)\,X +(\xi(\mu)-\mu)\eta(X)\xi
\end{equation}
for any $X\in\mathfrak{X}(M)$. Next, taking inner product of \eqref{3.27KK} with $\xi$ and using \eqref{2.4}, we obtain
$X(\lambda+\mu+f)=\xi(\lambda+\mu+f)\eta(X)$. Putting this into \eqref{3.27KK}, we achieve
\begin{equation}\label{3.27KKK}
 QX=\big(\xi(\lambda+f)-\mu-2n\big)X-\big(\xi(\lambda+f)+\mu\big)\eta(X)\xi
\end{equation}
for any $X\in\mathfrak{X}(M)$. This shows that $(M,g)$ is an $\eta$-Einstein manifold. Further, contracting \eqref{3.25} over $X$ with respect to an orthonormal basis $\{e_i\}_{1\le i\le2n+1}$, we compute
\begin{equation}\label{ZZ}
 {\rm Ric}(Y,Df)=-\sum\nolimits_{\,i=1}^{2n+1}g((\nabla_{e_{i}}Q)Y,e_{i}) + Y(r) +2n\,Y(\lambda) +Y(\mu)-\eta(Y)\xi(\mu) +2\,n\mu\,\eta(Y).
\end{equation}
The following formula for Riemannian manifolds is well known:
${\rm trace}_g\{X\longrightarrow (\nabla_XQ)Y\}=\frac{1}{2}\,Y(r)$.
Applying this formula in \eqref{ZZ}, we achieve
\begin{equation}\label{3.26}
 {\rm Ric}(Y,Df)=\frac{1}{2}\,Y(r)+2n\,Y(\lambda)+Y(\mu)-\eta(Y)\xi(\mu)+2\,n\mu\,\eta(Y)
\end{equation}
for any $X\in\mathfrak{X}(M)$. From \eqref{2.4} one can compute ${\rm Ric}(\xi, Df)=-2\,n\,\xi(f)$. Plugging it into \eqref{3.26}, we find
$\xi(r)+4n\{\xi(f+\lambda)+\mu\}=0$.
Using this in the trace of \eqref{3.1}, we get $\xi(\lambda+f)=2n+1-\mu+\frac{r}{2n}$.
By virtue of this, equation \eqref{3.27KKK} reduces to
\begin{equation}\label{3.29M}
 QX = (1+\frac{r}{2n})X -\big((2n+1)+\frac{r}{2n}\big)\eta(X)\xi
\end{equation}
for any $X\in\mathfrak{X}(M)$. By our assumptions, $\xi(r)=0$, therefore, the trace of \eqref{3.1} gives $r = -2n(2n+1)$.
Thus, from (\ref{3.29M}) the required result follows.
\end{proof}



Next, considering a Kenmotsu metric as an almost $\eta$-Ricci soliton, whose non-zero potential vector field $V$ is pointwise
collinear to the Reeb vector field $\xi$, we extend Theorem \ref{thm4.1} from gradient almost $\eta$-Ricci solitons to almost $\eta$-Ricci solitons by proving the following.

\begin{theorem}\label{thm4.2}
If a Kenmotsu manifold $M^{2n+1}(\varphi, \xi, \eta, g)$ admits an almost $\eta$-Ricci soliton with non-zero potential vector field $V$ collinear to the Reeb vector field $\xi$ and $\xi$ leaves the scalar curvature $r$ invariant, then $(M,g)$ is an Einstein manifold with constant scalar curvature $r=-2n(2n+1)$.
\end{theorem}

\begin{proof} Since
$V=\tau\,\xi$ for some smooth function $\tau$ on $M$, it follows that
\begin{eqnarray*}
 ({\cal L}_V\,g)(X,Y)=X(\tau)\eta(Y)+Y(\tau)\eta(X) +2\tau\big(g(X,Y)-\eta(X)\eta(Y)\big)
\end{eqnarray*}
for any $X,Y\in\mathfrak{X}(M)$. By virtue of this, the soliton equation \eqref{1.1} transforms into
\begin{equation}\label{3.30}
 2{\rm Ric}(X,Y)+X(\tau)\eta(Y)+Y(\tau)\eta(X) +2(\tau+\lambda) g(X,Y)=2(\tau-\mu)\eta(X)\eta(Y)
\end{equation}
for any $X,Y\in\mathfrak{X}(M)$. Now, putting $X=Y=\xi$ in \eqref{3.30} and using \eqref{2.5}, yields $\xi(\tau) = 2n -\lambda-\mu$. Thus, \eqref{3.30} yields $X(\tau)=2n-\lambda-\tau$. Using this in \eqref{3.30} implies that
\begin{equation}\label{3.31}
 {\rm Ric} = -(\tau+\lambda)\,g -(2n-\tau-\lambda)\,\eta\otimes\eta.
\end{equation}
Hence, $(M,g)$ is $\eta$-Einstein.
Moreover, if
$\xi$ leaves the scalar curvature $r$ invariant, i.e., $\xi(r) = 0$, again, tracing \eqref{3.2} gives
$\xi(r)=-2\{r+2n(2n+1)\}$, and therefore, $r=-2n(2n+1)$. Using this in the trace of \eqref{3.31} yields $\tau+\lambda=2n$.
By \eqref{3.31}, $QX=-2nX$; thus, $(M,g)$ is an Einstein manifold.
\end{proof}


If $\tau$ is constant (instead of being a function) and $V=\tau\xi$, then \eqref{3.30} and \eqref{3.31} also hold. Setting $X=Y=\xi$ in \eqref{3.30} and using \eqref{2.5}, gives $\xi(\tau) = 2n -\lambda-\mu$. This yields $\lambda+\tau=2n$ when $\tau$ is constant.
Hence, without assuming that $\xi$ leaves the scalar curvature $r$ invariant, from \eqref{3.31} we conclude that $(M,g)$ is an Einstein manifold with the Einstein constant $-2n$, therefore, we derived the following.

\begin{theorem}
If a Kenmotsu manifold $M^{2n+1}(\varphi, \xi, \eta, g)$ admits a non-trivial almost $\eta$-Ricci soliton with $V = \tau\,\xi$ for some constant $\tau$, then it is an Einstein manifold with constant scalar curvature $-2n(2n+1)$.
\end{theorem}

Now, we construct explicit example of a five-dimensional Kenmotsu manifold that admits an $\eta$-Ricci soliton and a gradient $\eta$-Ricci soliton.

\begin{example}\rm
Let $M^5=\{(x,y,z,u,v)\in\mathbb{R}^5\}$ be a $5$-dimensional manifold, where $(x,y,z,u,v)$ are Cartesian coordinates in $\mathbb{R}^5$.
Let $e_1=v\frac{\partial}{\partial x},\, e_2=v\frac{\partial}{\partial y}$, $e_3=v\frac{\partial}{\partial z}$, $e_4=v\frac{\partial}{\partial u}$, $e_5=-v\frac{\partial}{\partial v}$.
Clearly, $(e_i)$ form an orthonormal basis of vector fields on $M^5$. Define the structure $(\varphi,\xi,\eta,g)$ as follows:
\[
 \varphi(e_1)=e_2,\quad \varphi(e_2)=-e_1,\quad \varphi(e_3)=e_4,\quad \varphi(e_4)=-e_3,\quad \varphi(e_5)=0,\quad \xi=e_5,\quad \eta=dv,
\]
and $g_{ij}=\delta_{ij}$ -- the Kronecker symbol.
By Koszul's formula, we compute the non-trivial components of the Levi-Civita connection $\nabla$ in the following form:
\begin{equation}\label{4.20}
\nabla_{e_i}\,{e_i}=-e_5,\quad \nabla_{e_i}\,{e_5}=e_i\quad (i=1,2,3,4).
\end{equation}
Using this, we can verify that $M^5(\varphi,\xi,\eta,g)$ is a Kenmotsu manifold
and compute the following non-zero components of its curvature tensor:
\[
\begin{tabular}{|c|c|c|c|c|}
 \hline
 $R(e_1,e_2)e_1=e_2$ & $R(e_1,e_2)e_2=-e_1$ & $R(e_1,e_3)e_1=e_3$ & $R(e_1,e_3)e_3=-e_1$\\
 \hline
 $R(e_1,e_4)e_1=e_4$ & $R(e_1,e_4)e_4=-e_1$ & $R(e_1,e_5)e_1=e_5$ & $R(e_1,e_5)e_2=e_5$\\
 \hline
 $R(e_1,e_5)e_5=e_1$ & $R(e_2,e_3)e_2=e_3$ & $R(e_2,e_3)e_3=-e_2$ & $R(e_2,e_4)e_2=e_4$\\
 \hline
 $R(e_2,e_4)e_2=-e_3$ & $R(e_2,e_5)e_5=-e_2$ & $R(e_3,e_4)e_3=e_4$ & $R(e_3,e_4)e_4=-e_3$\\
 \hline
 $R(e_3,e_5)e_3=e_5$ & $R(e_3,e_5)e_5=-e_3$ & $R(e_4,e_5)e_4=e_5$ & $R(e_4,e_5)e_4=-e_5$\\
 \hline
\end{tabular}
\]
Thus, the nonzero components of the Ricci tensor are as follows: ${\rm Ric}(e_i,e_i)=-4$ for $i=1,\ldots,5$; therefore, the Ricci tensor is given by
\begin{equation}\label{4.22}
 {\rm Ric}=-4\,g .
\end{equation}
Hence, $(M,g)$ is an Einstein manifold with constant scalar curvature $r=-20=-2n(2n+1)$ for $n=2$.
If we consider the potential vector field $V=\sigma\xi$ for some constant $\sigma$, then, using \eqref{4.20}, we~get
\begin{equation}\label{4.21}
 {\cal L}_V\,g =2\,\sigma\,g -2\,\sigma\,\eta\otimes\eta.
\end{equation}
Hence, using \eqref{4.22} in \eqref{4.21} and remembering the soliton equation \eqref{1.1}, we can say that the metric $g$ is an $\eta$-Ricci soliton with the potential vector field $V=\sigma\xi$ and the constants $\lambda=4-\sigma$ and $\mu=\sigma$.
Again, $\lambda+\mu=4=2n$, where $n=2$. This illustrates Lemma~\ref{lem3.3}.

In general, for the potential vector field
$V=2x\frac{\partial}{\partial x}+2y\frac{\partial}{\partial y}+2z\frac{\partial}{\partial z}
+2u\frac{\partial}{\partial u}+v\frac{\partial}{\partial v}$ on $M^5$, using \eqref{4.20}, we find
  $({\cal L}_V\,g)(e_i,e_j)=\Big\{\begin{array}{cc}
    2, & \text{if \ $i=j=1,2,3,4$},\\
    0, & \text{elsewhere};
  \end{array}$
therefore, we achieve
\begin{equation}\label{3.52}
 {\cal L}_V\,g = 2\,g-2\,\eta\otimes\eta.
\end{equation}
So, combining \eqref{4.22} and \eqref{3.52}, one can see that soliton equation \eqref{1.1} is satisfied by $\lambda=3$ and $\mu=1$,
i.e., the metric $g$ is an $\eta$-Ricci soliton with the above potential vector field $V$ and the constants $\lambda=3$ and $\mu=1$,
which also satisfy $\lambda + \mu=2n$, for $n=2$. Choosing the function $f(x,y,z)=x^2+y^2+z^2+u^2+\frac12\,v^2$ on $M^5$,
from \eqref{4.22} and \eqref{3.52} we conclude that the metric $g$ is a gradient $\eta$-Ricci soliton with the potential function $f$.
\end{example}

\section{Conclusion}

In this article, we use methods of local Riemannian geometry to study
solutions of
\eqref{1.2}
and characterize Einstein metrics in such broader classes of metrics as almost $\eta$-Ricci solitons and $\eta$-Ricci solitons on Kenmotsu manifolds, which compose a special class of almost contact manifolds.
Our results generalize some results of other authors and
are important not only for differential geometry, but also for theoretical physics.
Following \cite{duggal2017almost}, we can think about physical applications of (almost) $\eta$-Ricci solitons.
%
We delegate for further study the following questions:

\noindent\ \
1: Under what conditions is an $\eta$-Ricci soliton with the potential vector field on a Kenmotsu manifold trivial?

\noindent\ \
2: Are Theorems 1--3 true without assuming the $\eta$-Einstein condition or that the potential vector field $V$ is an infinitesimal contact transformation or $V$ is collinear to the Reeb vector field?

\noindent\ \
3: Is the Kenmotsu metric, admitting a gradient almost $\eta$-Ricci soliton with a potential vector field, trivial?

\noindent\ \
4: Which of the results of this paper are also true for generalized quasi-Einstein Kenmotsu manifolds?

%





\end{document}